\documentclass[11pt,letterpaper,upref]{amsart}
\usepackage[letterpaper,margin=1.5in]{geometry}
\usepackage{amssymb}
\usepackage{amsxtra}
\usepackage[mathscr]{eucal}
\usepackage{graphics}
\usepackage{mathtools}

\numberwithin{equation}{section}

\newtheorem{theorem}{Theorem}[section]
\newtheorem{proposition}[theorem]{Proposition}
\newtheorem{lemma}[theorem]{Lemma}
\newtheorem{corollary}[theorem]{Corollary}

\theoremstyle{definition}

\newtheorem{definition}[theorem]{Definition}
\newtheorem{example}[theorem]{Example}

\newtheorem{remark}[theorem]{Remark}

\newcommand{\af}{\alpha_f}
\newcommand{\ahat}{\widehat{a}}
\newcommand{\al}{\alpha} 
\newcommand{\ardma}{\al_{\rd/\fa}}
\newcommand{\ardmf}{\al_{\rd/\<f\>}}
\newcommand{\aut}{\operatorname{aut}}

\newcommand{\bk}{\mathbf{k}}
\newcommand{\bm}{\mathbf{m}} 
\newcommand{\bn}{\mathbf{n}}
\newcommand{\bo}{\boldsymbol{\omega}}
\newcommand{\bq}{\mathbf{q}}
\newcommand{\bs}{\mathbf{s}} 
 
\newcommand{\bu}{\mathbf{u}} 

\newcommand{\bzero}{\boldsymbol{0}}
\newcommand{\CC}{\mathbb{C}}
\newcommand{\Da}{\Delta_{\al}}

\newcommand{\Daf}{\Delta_{\al_f}}
\newcommand{\Dafs}{\Delta_{\al_f}^1}
\newcommand{\Del}{\Delta}
\newcommand{\dist}{\operatorname{\mathsf{dist}}}
\newcommand{\DGj}{\Del_{\G,j}}
\newcommand{\dq}{\mathsf{d}_Q}
\newcommand{\dg}{\mathsf{d}_{\G}}
\newcommand{\dT}{\ensuremath|\hspace{-2.1pt}|}
\newcommand{\eps}{\varepsilon}
\newcommand{\fa}{\mathfrak{a}} 

\newcommand{\fg}{\mathfrak{g}}

\newcommand{\fhat}{\widehat{f}}
\newcommand{\ghat}{\widehat{g}}
\newcommand{\fhs}{\widehat{f^*}}
\newcommand{\ghs}{\widehat{g^*}}
\newcommand{\Fix}{\operatorname{\mathsf{Fix}}}
\newcommand{\FG}{\Fix_{\G}}
\newcommand{\FGn}{\Fix_{\G_n}}
\newcommand{\FGno}{\Fix_{\G_n}^\circ}
\newcommand{\FGo}{\Fix_{\G}^\circ}
\newcommand{\G}{\Gamma}
\newcommand{\Gn}{\G_n}
\newcommand{\h}{\mathsf{h}} 
\newcommand{\hhat}{\widehat{h}}

\newcommand{\Hm}{H_{\bm}}
\newcommand{\Hmj}{H_{\bm_j}}
\newcommand{\lam}{\lambda}
\newcommand{\lizdc}{\ell^\infty(\zd,\CC)}
\newcommand{\lizdr}{\ell^\infty(\zd,\RR)}
\newcommand{\lozdr}{\ell^1(\zd,\RR)}
\newcommand{\lozdc}{\ell^1(\zd,\CC)}
\newcommand{\lizdz}{\ell^\infty(\zd,\ZZ)}
\newcommand{\lzdgc}{\ell(\zd/\G,\CC)}

\newcommand{\lzdgr}{\ell(\zd/\G,\RR)}
\newcommand{\lzddr}{\ell(\zd/\Delta,\RR)}
\newcommand{\m}{\mathsf{m}}
\newcommand{\mf}{\mathfrak{m}_f}

\newcommand{\nf}{\mathfrak{n}_f}

\newcommand{\om}{\omega}
\newcommand{\Om}{\Omega}
\newcommand{\OG}{\Om_{\G}}
\newcommand{\OGn}{\Om_{\G_n}}

\newcommand{\p}{\mathsf{p}}
\newcommand{\sP}{\mathsf{P}}
\newcommand{\PG}{\sP_{\G}} 
\newcommand{\QQ}{\mathbb{Q}}
\newcommand{\rd}{R_d}
\newcommand{\rdma}{\rd/\fa}

\newcommand{\RR}{\mathbb R}
\newcommand{\sd}{\SS^{d}}
\newcommand{\SF}{\mathcal{F}}
\newcommand{\sig}{\sigma}
\newcommand{\sigt}{\tilde{\sigma}}
\renewcommand{\SS}{\mathbb{S}} 
\newcommand{\TT}{\mathbb{T}}
\newcommand{\tzd}{\TT^{\zd}}
\newcommand{\U}{\mathsf{U}} 
\newcommand{\V}{\mathsf{V}}

\newcommand{\vo}{v^{(\bo)}}

\newcommand{\xrdma}{X_{\rd/\fa}}
\newcommand{\xrdmf}{X_{\rd/(f)}}
\newcommand{\ZZ}{\mathbb{Z}} 
\newcommand{\zd}{\ZZ^{d}}

\renewcommand{\ge}{\geqslant}
\renewcommand{\le}{\leqslant}
\newcommand{\<}{\langle}
\renewcommand{\>}{\rangle}  
\renewcommand{\emptyset}{\varnothing}

\renewcommand\Re{\operatorname{Re}}

\renewcommand{\setminus}{\smallsetminus}

\begin{document}

\title[Homoclinic and Periodic Points]
{Homoclinic Points, Atoral Polynomials, \\ and
Periodic Points of Algebraic $\mathbb{Z}^d$-actions}

\author{Douglas Lind}

\address{Douglas Lind: Department of Mathematics, University of
  Washington, Seattle, Washington 98195, USA}
  \email{lind@math.washington.edu}

\author{Klaus Schmidt}

\address{Klaus Schmidt: Mathematics Institute, University of Vienna, 
Nordberg\-stra{\ss}e 15, A-1090 Vienna, Austria}
\email{klaus.schmidt@univie.ac.at}

\author{Evgeny Verbitskiy}

\address{Evgeny Verbitskiy: Mathematical Institute, University of
Leiden, PO Box 9512, 2300 RA Leiden, The Netherlands \newline\indent
\textup{and} \newline\indent Johann Bernoulli Institute for Mathematics
and Computer Science, University of Groningen, PO Box 407, 9700 AK,
Groningen, The Netherlands} 
\email{evgeny@math.leidenuniv.nl}

\date{\today}




\begin{abstract}
   Cyclic algebraic $\zd$-actions are defined by ideals of Laurent
   polynomials in $d$ commuting variables. Such an action is expansive
   precisely when the complex variety of the ideal is disjoint from the
   multiplicative $d$-torus. For such expansive actions it is known that
   the limit for the growth rate of periodic points exists and is equal to
   the entropy of the action. In an earlier paper the authors extended
   this result to ideals whose variety intersects the $d$-torus in a
   finite set. Here we further extend it to the case when the
   dimension of intersection of the variety with the $d$-torus is at
   most $d-2$. The main tool is the construction of homoclinic points
   which decay rapidly enough to be summable.
\end{abstract}

\maketitle

\section{Introduction}\label{sec:introduction}

An \emph{algebraic $\zd$-action} on a compact abelian group $X$ is a
homomorphism $\al\colon\zd\to\aut(X)$ from $\zd$ to the group of
(continuous) automorphisms of $X$.  We denote the image of $\bn\in\zd$
under $\al$ by $\al^{\bn}$, so that
$\al^{\bm+\bn}=\al^{\bm}\circ\al^{\bn}$ and
$\al^{\mathbf{0}}=\text{Id}_{X}$.

We will consider here cyclic algebraic $\zd$-actions, described as
follows, using notation and terminology from \cite{LSV}. Let
$\rd=\ZZ[u_1^{\pm1},\dots,u_d^{\pm1}]$ denote the ring of Laurent
polynomials with integer coefficients in the commuting variables
$u_1,\dots,u_d$. We write $f\in\rd$ as
$f=\sum_{\bm\in\zd}f_{\bm}\bu^{\bm}$, where $\bu=(u_1,\dots,u_d)$,
$\bm=(m_1,\dots,m_d)\in\zd$, $\bu^{\bm}=u_1^{m_1}\dots u_d^{m_d}$, and
$f_{\bm}\in\ZZ$ with $f_{\bm}=0$ for all but finitely many $\bm$.

Let $\TT=\RR/\ZZ$, and define the shift $\zd$-action $\sig$ on $\tzd$ by
\begin{displaymath}
   (\sig^{\bm}x)_{\bn}=x_{\bm+\bn}
\end{displaymath}
for $\bm\in\zd$ and $x=(x_{\bn})\in\tzd$. For $f=\sum
f_{\bm}\bu^{\bm}\in\rd$ we put
\begin{displaymath}
   f(\sig)=\sum_{\bm\in\zd}f_{\bm}\sig^{\bm}\colon \tzd\to\tzd.
\end{displaymath}

We identify $\rd$ with the dual group of $\tzd$ by setting
\begin{displaymath}
   \<f,x\>=e^{2\pi i f(\sig)(x)_{\mathbf{0}}} =e^{2 \pi i \sum_{\bm}f_{\bm}x_{\bm}}
\end{displaymath}
for $f\in\rd$ and $x\in\tzd$. In this identification the shift
$\sig^{\bm}$ is dual to multiplication by $\bu^{\bm}$ on $\rd$.
A closed subgroup $X\subset\tzd$ is shift-invariant if and only if its
annihilator
\begin{displaymath}
   X^{\perp}=\{h\in\rd:\<h,x\>=1\text{\ for every $x\in X$} \}
\end{displaymath}
is an ideal in $\rd$. In view of this we write, for every 
ideal $\fa$ in $\rd$,
\begin{displaymath}
   \xrdma=\fa^{\perp}=\{x\in\tzd:\<h,x\>=1 \text{\ for every $h\in\fa$}\}
\end{displaymath}
for the closed, shift-invariant subgroup of $\tzd$ annihilated by
$\fa$. Here the dual group of $\xrdma$ is $\rdma$. Denote by $\ardma$
the restriction of the shift-action $\sig$ on $\tzd$ to $\xrdma$. A
\emph{cyclic algebraic $\zd$-action} is one of this form, corresponding to
the cyclic $\rd$-module $\rdma$.

According to \cite[Eqn.\ (1-1)]{LSW} or \cite[Thm.\ 18.1]{DSAO}, the
topological entropy of $\ardma$, which coincides with its entropy with
respect to Haar measure on $\xrdma$, is given by
\begin{equation}
   \label{eq:entropy}
   \h(\ardma)= 
   \begin{cases}
      \infty &\text{if $\fa=\{0\}$}, \\
      \m(f)  &\text{if $\fa=\<f\>=f\cdot\rd$ for some nonzero
      $f\in\rd$,}\\
      0      &\text{if $\fa$ is nonprincipal,}
   \end{cases}
\end{equation}
where 
\begin{displaymath}
   \m(f)=\int_0^1\dots\int_0^1\,\log|f(e^{2\pi i t_1},\dots,e^{2\pi i
   t_d})|\,dt_1\dots dt_d
\end{displaymath}
is the \emph{logarithmic Mahler measure} of $f$.

An algebraic $\zd$-action $\al$ on a compact abelian group $X$ is
\emph{expansive} if there is a neighborhood $U$ of $0_X$ such that
$\bigcap_{\bm\in\zd}\al^{\bm}(U)=\{0_X\}$. To characterize expansiveness
for cyclic actions $\ardma$, let $\CC^\times$ denote $\CC\setminus\{0\}$
and let
\begin{displaymath}
   \V(\fa)=\{(z_1,\dots,z_d)\in (\CC^\times )^d :g(z_1,\dots,z_d)=0
   \text{\ for all $g\in\fa$}\}
\end{displaymath}
denote the complex variety of the ideal $\fa$. Put
$\SS=\{z\in\CC:|z|=1\}$, so that $\SS^d$ is the unit multiplicative
$d$-torus in $(\CC^\times)^d$. Define the \emph{unitary variety} of
$\fa$ as
\begin{displaymath}
   \U(\fa)=\V(\fa)\cap\SS^d =\{(z_1,\dots,z_d)\in\V(\fa):|z_1|=\dots=|z_d|=1\}.
\end{displaymath}
According to \cite[Thm.\ 6.5]{DSAO}, $\ardma$ is expansive if and only
if $\U(\fa)=\emptyset$. 

In order to describe periodic points for $\ardmf$, let $\SF$ denote the
collection of finite-index subgroups of $\zd$, and let $\G$ be an
arbitrary element of $\SF$. Define
\begin{displaymath}
   \<\G\>=\min\{\|\bm\|:\mathbf{0}\ne\bm\in\G\},
\end{displaymath}
where
$\|\bm\|=\max\{|m_1|,\dots,|m_d|\}$. A point $x\in X$ has \emph{period
$\G$} if $\al^{\bm}x=x$ for all $\bm\in\G$. Let
\begin{displaymath}
   \FG(\ardma)=\{x\in\xrdma: x \text{\ has period $\G$}\}
\end{displaymath}
be the closed subgroup of $\xrdma$ consisting of all $\G$-periodic
points. Since $\FG(\ardma)$ may be infinite (examples are given in
the next section), we reduce it to a finite object by
forming the quotient of $\FG(\ardma)$ by its connected component
$\FGo(\ardma)$ of the identity. We therefore define
\begin{displaymath}
   \PG(\ardma)=|\FG(\ardma)/\FGo(\ardma)|,
\end{displaymath}
where $|\cdot|$ denotes cardinality. 

Define the \emph{upper logarithmic growth rate} of the number of
periodic components of $\ardma$ as
\begin{equation}
   \label{eq:components}
   \p^+(\ardma)\coloneqq\limsup_{\<\G\>\to\infty}\,\frac1{|\zd/\G|}\,
   \log \PG(\ardma).
\end{equation}
For an arbitrary nonzero ideal $\fa$ it was proved in \cite[Thm.\
21.1]{DSAO} that $ \p^+(\ardma)=\h(\ardma)$. For expansive actions much
more is known.

\begin{theorem}
   \label{th:expansive-case}
   Let $\fa$ be an ideal for which $\ardma$ is expansive, or
   equivalently for which $\U(\fa)=\emptyset$. The the $\limsup$ in
   \eqref{eq:components} is actually a limit, i.e.,
   \begin{equation}
      \label{eq:limit-exists}
      \lim_{\<\G\>\to\infty}\,\frac1{|\zd/\G|}\,
      \log \PG(\ardma)=\h(\ardma).
   \end{equation}
\end{theorem}

It is not known whether \eqref{eq:limit-exists} holds for \emph{all}
cyclic actions with nonzero $\fa$ (when $\fa=0$ the left side is 0 and
the right side is $\infty$). Even when $d=1$ the existence of this limit
involves a deep diophantine estimate due to Gelfond \cite{G} (see
Section \ref{sec:diophantine}, or \cite[Sec.\ 4]{L} for details). In
\cite{LSV} we proved the following partial result.

\begin{theorem}
   \label{th:finite-variety}
   Let $d\ge2$ and $\fa$ be an ideal in $\rd$ whose unitary variety
   $\U(\fa)$ is a finite set. Then \eqref{eq:limit-exists} holds.
\end{theorem}

The proof of Theorem \ref{th:finite-variety} in \cite{LSV} depends on
the crucial fact that if $\U(\fa)$ is finite, then all of its points
have coordinates which are algebraic numbers. With this information
Theorem \ref{th:finite-variety} can be deduced from Gelfond's
estimate. However, the route taken in \cite{LSV} is different, and
bypasses the diophantine issue: the algebraicity of the points in
$\U(\fa)$ allows the construction of certain well-behaved (so-called
\emph{summable}) homoclinic points of $\ardma$, which are then used to
prove not only Theorem \ref{th:finite-variety} but also a strong
specification property of the action $\ardma$ which is of independent
interest. We remark in passing that this approach involving summable
homoclinic points first appeared in \cite{SV} in the construction of
symbolic covers for a special class of cyclic $\zd$-actions, the
\emph{harmonic} actions.

In order to state our main result, we need to recall some notions from
real algebraic geometry. Let $z_j=x_j+iy_j$, and consider $\CC^d$ as
$\RR^{2d}$ with coordinates $(x_1,y_1,\dots,x_d,y_d)$. Then $\SS^d$ as a
subset of $\RR^{2d}$ is the real algebraic set defined by the equations
$x_j^2+y_j^2-1=0$ for $1\le j\le d$. For
$f(z_1,\dots,z_d)\in\CC[z_1,\dots,z_d]$, we can expand $f$ into its real
an imaginary polynomial parts, so that
\begin{equation}
   \label{eqn:real-expansion}
   f(x_1+iy_1,\dots,x_d+iy_d)=
   f_1(x_1,y_1,\dots,x_d,y_d)+if_2(x_1,y_1,\dots,x_d,y_d), 
\end{equation}
where $f_1,f_2\in\RR[x_1,y_1,\dots,x_d,y_d]$ have real
coefficients. Then $\U(f)$ is the real algebraic set defined by the
equations $x_j^2+y_j^2-1=0$ for $1\le j\le d$ together with two further
equations $f_1(x_1,y_1,\dots,x_d,y_d)=0$ and
$f_2(x_1,y_1,\dots,x_d,y_d)=0$. This discussion extends to Laurent
polynomials $f\in\rd$ by observing that there is a $\bm\in\zd$ with
$\bu^{\bm}f(\bu)\in\ZZ[u_1,\dots,u_d]$ and that $\U(\bu^{\bm}f)=\U(f)$.
Every ideal $\fa$ is finitely generated, say by $g_1,\dots,g_r$. Hence
its unitary variety $\U(\fa)=\U(g_1)\cap\dots\cap\U(g_r)$ is again a
real algebraic subset of $\RR^{2d}$.

The cell decomposition theorem for real semialgebraic sets \cite[Thm.\
2.11]{vdD} therefore applies, and so $\U(\fa)$ can be written as a
finite disjoint union of open cells of various dimensions. The
\emph{dimension} of $\U(\fa)$ is defined to be the maximum dimension of
these cells, and this number is the same for all possible cell
decompositions. By convention we define $\dim\emptyset=-\infty$.
There is a finite algorithm \cite[Algorithm 14.10]{BPR} for computing
the dimension of a real semialgebraic set defined by polynomials with
rational coefficients. 

Our main result is the following extension of Theorem
\ref{th:finite-variety}.

\begin{theorem}
   \label{th:main}
   Let $d\ge2$ and let $\fa$ be an ideal in $\rd$. If the dimension of
   $\U(\fa)$ is at most $d-2$, then
   \begin{displaymath}
       \lim_{\<\G\>\to\infty}\,\frac1{|\zd/\G|}\,
      \log \PG(\ardma)=\h(\ardma).
   \end{displaymath}
\end{theorem}

In \cite{LSV} it is explained how to use the algebraic machinery in
\cite[Sec.\ 21]{DSAO} to reduce the proof of Theorem \ref{th:main} to the
case when $\fa$ is prime and principal. In view of this we assume from
now on that $\fa=(f)=f \rd$ for some nonzero irreducible Laurent
polynomial $f\in\rd$. In this case we call the shift-action
$\af\coloneqq\ardmf$ on the group $X_f\coloneqq \xrdmf$ a
\emph{principal algebraic $\zd$-action}.

Let $\Om$ denote the set of torsion points in $\SS^d$, so the
coordinates of points in $\Om$ are roots of unity. For $\G\in\SF$ let
\begin{displaymath}
   \OG\coloneqq\{\bo\in\Om:\bo^{\bm}=\om_1^{m_1}\cdots\om_d^{m_d}=1
   \text{ for every $\bm\in\G$}\}.
\end{displaymath}
Basic duality shows that $\OG$ is the dual group of $\zd/\G$, and so
$|\OG|=|\zd/\G|$.

As mentioned above, according to \cite[Thm.\ 21.1]{DSAO},
\begin{displaymath}
   \p^+(\af):=\limsup_{\<\G\>\to\infty}\frac1{|\zd/\G|}
   \log\PG(\af)=\h(\af)=\m(f)=
   \int_{\SS^d}\log|f(\bs)|\,d\lambda(\bs),
\end{displaymath}
where $\lambda$ is normalized Lebesgue measure on $\SS^d$. In 
\cite[Lemma 2.1]{LSV} we claimed that
\begin{equation}
   \label{eq:perptcalcerror}
   \PG(\af)=\prod_{\bo\in\OG\smallsetminus\U(f)}|f(\bo)|.
\end{equation}
However, using the notation from the proof there, we in fact need to
divide the right-hand side of \eqref{eq:perptcalcerror} by
\begin{displaymath}
   c_{\Gamma}(f)=|f(\widetilde{\sigma})(V_{\Gamma}(\ZZ))/f(\widetilde{\sigma})
   (V_{\Gamma}'(\ZZ))|.
\end{displaymath}
The proofs of the main results in both \cite{LSV} and here do not depend
on \eqref{eq:perptcalcerror}, and show that
\begin{displaymath}
   \frac1{|\zd/\G|}\log c_{\Gamma}(f)\to0\text{\quad as \quad}\<\G\>\to\infty,
\end{displaymath}
so that  \eqref{eq:perptcalcerror} is asymptotically correct.

Observe that since $|\zd/\G|=|\OG|$, we are dealing with sums of the form
\begin{equation}
   \label{eq:riemann-sum}
   \frac1{|\OG|} \sum_{\bo\in\OG\smallsetminus\U(f)}\log|f(\bo)|,
\end{equation}
which are Riemann sum approximations to
$\int_{\SS^d}\log|f|\,d\lambda=\m(f)$. Hence proving the existence of
the limit in Theorem \ref{th:main} is exactly the same as proving that
these Riemann sums for $\log|f|$ converge to its integral over $\SS^d$ as
$\<\G\>\to\infty$. In trying to prove convergence of these Riemann sums
one encounters two problems.

The first problem involves the omission of summands in
\eqref{eq:riemann-sum} with $\bo\in\OG\cap\U(f)$. As we will see, each
such $\bo$ contributes one dimension to $\FGo(\af)$, so that $\FGo(\af)$
is a torus of dimension $|\OG\cap\U(f)|$.  This omission is necessary, of course,
since any summand with $f(\bo)=0$ would contribute~ $-\infty$ to the
Riemann sum. This situation was easily dealt with in the case $\U(f)$ is
finite \cite{LSV} by observing that the dimension of this torus is then
bounded, so there is an easy bound on the number of points in any
separated set of any closed subgroup. However, here the dimension of
this torus can be unbounded (see Example \ref{exam:1+x+y+z}). We control
this by invoking a result of Mann \cite{Mann} that implies that all torsion points
in $\U(f)$ lie in a finite union of cosets of rational subtori, and this
provides a sufficiently uniform estimate for separated sets.

The second problem involves those $\bo\in\OG\smallsetminus\U(f)$ which
may be very close to $\U(f)$, and is much more serious. For these
points the value $|f(\bo)|$ will be extremely small, but nonzero, and so
$\log|f(\bo)|$ could conceivably take such a large negative value that
the average value in \eqref{eq:riemann-sum} is significantly less than
$\m(f)$. Can this happen for a sequence of $\G_n$ with
$\<\G_n\>\to\infty$? This problem is essentially a multidimensional
version of the diophantine problem mentioned earlier.

Unfortunately this diophantine problem is unsolved in the generality
required here, which forces us to impose the additional hypothesis that
$\U(f)$ has dimension at most $d-2$. This hypothesis turns out to be
equivalent to the existence of summable homoclinic points (Theorem
\ref{thm:summable-points}), and allows us to extend the approach
developed in \cite{SV} and \cite{LSV} to this more general situation,
thereby proving Theorem \ref{th:main}. As a side benefit of this
approach, we obtain in Corollary \ref{cor:lowerbound} a diophantine
estimate concerning the proximity of torsion points to the variety
$\U(f)$ under the hypothesis that $\dim \U(f)\le d-2$. This result is
completely analogous to Gelfond's estimate. In the case when $\dim \U(f)=
d-1$ (which when $d=1$ reduces to the setting of Gelfond's estimate),
our approach yields a slightly weaker estimate (Corollary
\ref{cor:weak-gelfond}).

The question of whether entropy coincides with the logarithmic growth
rate of the number of connected components in $\FG(\af)$ as
$\<\G\>\to\infty$ irrespective of whether the principal action is
expansive or not is a special case of a much more general problem. For a
principal algebraic $\zd$-action $\af$ the entropy was proved in
\cite{LSW} to be equal to the logarithmic Mahler measure of the Laurent
polynomial $f\in R_d$, and to the upper logarithmic growth rate of
$\PG(\af)$ in \cite{DSAO}. For an expansive principal algebraic action
$\af$ of an arbitrary residually finite countable amenable group $G$,
the entropy $\h(\af)$ was identified in \cite{DS} with an appropriately
defined logarithmic growth rate of the number of points with finite
orbits of $\af$ and, as a consequence, with the Fuglede-Kadison
determinant associated with the element $f$ in the integral group ring
$\ZZ[G]$, acting on the group von Neumann algebra
$\mathcal{N}[G]$. Without the hypothesis of expansiveness, the
relationship between entropy, logarithmic growth rate of periodic points
(or periodic components), and Fuglede-Kadison determinants is still
completely open. In this sense, the equality of $\p^+(\af)$ and
$\h(\af)$, and the more precise version in Theorem \ref{th:main} are of
interest, since they provide links between these objects without the
hypothesis of expansiveness (but obviously under very special
circumstances).

The authors are grateful to the Erwin Schr\"{o}dinger International
Institute for Mathematical Physics, the Max Planck Institute for
Mathematics, FWF Grant S9613, the University of Washington Mathematics
Department, the Lorentz Center in Leiden, and Microsoft Research for
their generous support of this work.

\section{Atoral polynomials}
\label{sec:homoclinic-atoral}

Here we characterize those polynomials for which our techniques apply.
Recall that the units of $\rd$ are $\pm\bu^{\bn}$, where $\bn$ is an
arbitrary element of $\zd$. We say that a Laurent polynomial in $\rd$ is
\emph{irreducible} if it is not a unit and if it has no factorizations
apart from units.

The basic notion of atorality is motivated by the paper of
Agler, McCarthy, and Stankus \cite{AMS}.

\begin{definition}
   An irreducible Laurent polynomial $f\in R_d$ is called \emph{toral}
   if $h$ is in the ideal $(f)$ whenever $h$ is in $R_d$ with
   $\U(h)\supseteq\U(f)$. Otherwise $f$ is called \emph{atoral}. A
   general Laurent polynomial is called \emph{atoral}
   if each of its irreducible factors is atoral; otherwise it is called
   \emph{toral}. 
\end{definition}

In \cite{AMS} the notion of atorality of polynomials with complex
coefficients is introduced and studied. One main result there is that
such a polynomial is atoral if and only if its unitary variety is
contained in an algebraic set of (complex) dimension $d-2$. Our setting,
involving polynomials with rational coefficients and a corresponding
notion of irreducibility, is somewhat different. However, we obtain a
similar result, whose proof follows closely the spirit of \cite{AMS}.

\begin{proposition}
   \label{prop:atoral}
   Let $f\in R_d$ be irreducible. Then $f$ is atoral if and only if
   $\dim \U(f)\le d-2$.
\end{proposition}

Symmetry plays a crucial role in the proof.

\begin{definition}
   Let $f(\bu)=\sum f_{\bn}\bu^{\bn}\in\rd$. The \emph{adjoint} of $f$
   is $f^*(\bu)=f(\bu^{-1})=\sum f_{-\bn}\bu^{\bn}$. We say that $f$ is
   \emph{symmetric} if $f^*(\bu)=f(\bu)$, and that $f$ is
   \emph{essentially symmetric} if there is an $\bm\in\zd$ such that
   $f^*(\bu)=\pm \bu^{\bm}f(\bu)$.
\end{definition}

Thus $f$ is symmetric if its array of coefficients is symmetric with
respect to the origin, and is essentially symmetric if this array is
symmetric (or skew-symmetric) with respect to the point
$\frac12\bm\in\RR^d$. 

\begin{lemma}
   \label{lem:symmetry}
   Let $f\in\rd$. Then $f$ is essentially symmetric if and only if
   $f^*\in(f)$. 
\end{lemma}

\begin{proof}
   If $f$ is essentially symmetric, then obviously
   $f^*(\bu)=\pm\bu^{\bm}f(\bu)\in(f)$. For the opposite implication we
   may assume that $f\ne0$. Suppose that $f^*=hf$ for some
   $h\in\rd$. Then $f=h^*f^*=h^*hf$, so that $(h^*h-1)f=0$. Since
   $f\ne0$, we obtain that $h^*h=1$. But the only units in $\rd$ are
   $\pm \bu^{\bm}$.
\end{proof}

\begin{lemma}
   \label{lem:fstar}
   Let $f\in\rd$. Then $\U(f^*)=\U(f)$.
\end{lemma}

\begin{proof}
   Since $f$ has real coefficients, for every $\bs\in\sd$ we have that
   $f^*(\bs)=\overline{f(\bs)}$. 
\end{proof}

\begin{lemma}
   \label{lem:codim2}
   Let $r\ge2$ and suppose that $g_1,\dots,g_r\in\QQ[u_1,\dots,u_d]$
   have no common factor. Let $\fg$ be the ideal in $\QQ[u_1,\dots,u_d]$
   they generate. Then $\dim\U(\fg)\le d-2$.
\end{lemma}

\begin{proof}
   We use induction on $d$. First suppose that $d=1$, and that
   $g_1,\dots,g_r\in\QQ[u_1]$ have no common factor. Then their greatest
   common divisor is $1$, so there are $A_j(u_1)\in\QQ[u_1]$ such that
   $\sum_{j=1}^r A_j(u_1)g_j(u_1)=1$. Hence
   $\U(\fg)=\bigcap_{j=1}^r\U(g_j)=\emptyset$, and
   $\dim\emptyset=-\infty\le -1$.

   To complete the induction argument, assume that the result is true for
   $d-1$, and let $g_1,\dots,g_r\in\QQ[u_1,\dots,u_d]$ have no common
   factor. Put $K=\QQ(u_1,\dots,u_{d-1})$ and consider
   \begin{displaymath}
      g_j(u_1,\dots,u_d)=\sum_k g_{jk}(u^{}_1,\dots,u^{}_{d-1})u_d^k
      \in K[u_d].
   \end{displaymath}
   By Gauss's Lemma, the greatest common divisor of the $g_j$ as
   elements in $K[u_d]$ is~ 1, so there are
   \begin{displaymath}
      a_j(u_1,\dots,u_d)=\sum_k a_{jk}(u^{}_1,\dots,u^{}_{d-1})u_d^k
   \end{displaymath}
   with $a_{jk}(u_1,\dots,u_{d-1})\in\QQ(u_1,\dots,u_{d-1})$ such that
   \begin{displaymath}
      \sum_{j=1}^r a_j(u_1,\dots,u_d)g_j(u_1,\dots,u_d)=1.
   \end{displaymath}
   Clearing denominators, we obtain
   \begin{displaymath}
      A_j(u_1,\dots,u_d)\in\QQ[u_1,\dots,u_d] \text{\quad and\quad}
      0\ne B(u_1,\dots,u_{d-1})\in\QQ[u_1,\dots,u_{d-1}]
   \end{displaymath}
   such that
   \begin{equation}
      \label{eq:gcdd}
      \sum_{j=1}^r A_j(u_1,\dots,u_d)g_j(u_1,\dots,u_d)=B(u_1,\dots,u_{d-1}).
   \end{equation}
   Define $\pi\colon\SS^d\to\SS^{d-1}$ by
   $\pi(s_1,\dots,s_d)=(s_1,\dots,s_{d-1})$. Since $B\ne0$, it follows
   that $\dim\U(B)\le d-2$ as a subset of $\SS^{d-1}$.
   Let $Y=\{\bs\in\U(B):g_{jk}(\bs)\ne0\text{ for some $j,k$}\}$ and
   $Z=\{\bs\in\U(B):g_{jk}(\bs)=0\text{ for all $j,k$}\}$. If
   $\bs\in Y$ then some $g_j(\bs,u_d)\ne0$, so that $\pi^{-1}(\bs)\cap
   Y$ is finite. Hence $\dim(\pi^{-1}(Y)\cap\U(\fg))\le\dim \U(B)\le
   d-2$. Observe that the $g_{jk}(u_1,\dots,u_{d-1})$ cannot have a common
   factor in $\QQ[u_1,\dots,u_{d-1}]$, since this would contradict our
   assumption on the $g_j$'s. Hence by the inductive hypothesis applied
   to the $g_{jk}$'s, we find that $Z$, being the unitary variety of the
   ideal generated by the $g_{jk}$, has dimension at most $d-3$. Hence
   $\dim \pi^{-1}(Z)\cap\U(\fg)\le d-2$. Thus $\U(\fg)$ is the union of
   the two semialgebraic sets $\pi^{-1}(Y)\cap\U(\fg)$ and
   $\pi^{-1}(Z)\cap\U(\fg)$, each of which has dimension $\le d-2$,
   completing the induction step, and the proof.
\end{proof}

\begin{proof}[Proof of Proposition \ref{prop:atoral}]
   First suppose that $f\in\rd$ is irreducible and atoral. Thus there is
   a $g\in\rd\smallsetminus (f)$ with $\U(g)\supseteq\U(f)$. Now $f$ and
   $g$ cannot have a common factor since $f$ does not divide $g$, so
   that by Lemma \ref{lem:codim2} applied to $g_1=f$ and $g_2=g$, we
   obtain that $\U(f)=\U(f)\cap\U(g)$ has dimension $\le d-2$.

   For the reverse implication, suppose that $f\in\rd$ is irreducible
   and that $\dim\U(f)\le d-2$.

   If $f$ is not essentially symmetric, then $g=f^*\notin(f)$ by Lemma
   \ref{lem:symmetry}, and $\U(g)=\U(f)$ by Lemma \ref{lem:fstar}, hence
   $f$ is atoral.

   Suppose now that $f$ is essentially symmetric, and that
   $f^*(\bu)=\bu^{\bm}f(\bu)$. Observe that if
   $f_2(u^{}_1,\dots,u^{}_d)=f(u_1^2,\dots,u_d^2)$, then $\U(f_2)$ is the
   union of $2^d$ smaller copies of $\U(f)$ and that
   $f_2^*(\bu)=\bu^{2\bm}f^{}_2(\bu)$. Hence by replacing $f$ with $f_2$
   we may assume that $\bm=2\bn\in2\zd$. Furthermore, replacing $f(\bu)$
   with $\bu^{\bn}f(\bu)$ (which preserves unitary varieties), we may
   assume that $f^*=f$ is exactly symmetric.

   We next show that the partial derivatives of $f$ must also vanish on
   $\U(f)$. Symmetry of $f$ means that $f$ is real-valued on $\sd$. Let
   $e\colon\TT^d\to\sd$ be the isomorphism $e(t_1,\dots,t_d)=(e^{2\pi i
   t_1},\dots,e^{2\pi i t_d})$. Then $f\circ e\colon\TT^d\to\sd$
   vanishes on $e^{-1}(\U(f))$. Let $\bs\in\U(f)$ and suppose that there
   is a $j$ for which $\partial f/\partial{u_j}(\bs)\ne0$. By the chain
   rule, $\partial(f\circ e)/\partial{t_j}(e^{-1}(\bs))\ne0$, hence the
   gradient of $f\circ e$ does not vanish at $e^{-1}(\bs)$. The Implicit
   Function Theorem shows that near $e^{-1}(\bs)$ the vanishing set of
   $f\circ e$ is $(d-1)$-dimensional, and hence $\U(f)$ is
   $(d-1)$-dimensional, contradicting our assumption.

   Hence all partials $\partial f/\partial u_j$ vanish on $\U(f)$. The
   case of constant $f$ is trivial, so we may assume that at least one
   $\partial f/\partial u_j\ne0$. But the $u_j$ degree of this partial
   is strictly less than the $u_j$-degree of $f$, and so cannot have a
   factor in common with $f$ by irreducibility of $f$. Hence using
   $g=\partial f/\partial u_j$ shows that $f$ is atoral.

   Finally, if $f$ is essentially symmetric with
   $f^*(\bu)=-\bu^{\bm}f(\bu)$, our previous simplifications show that
   we may assume that $f^*=-f$. But then $f$ is purely imaginary on
   $\sd$, so we can apply the preceding argument to $f/i$ and again
   obtain that $f$ is atoral.
\end{proof}

\begin{remark}
   Since $\U(f_1\cdots f_r)=\U(f_1)\cup\dots\cup\U(f_r)$, and since the
   dimension of the union of a finite number of semialgebraic sets
   equals the largest of the dimensions of those sets, it follows that
   Proposition \ref{prop:atoral} remains true for all $f\in\rd$.
\end{remark}

\section{Summable homoclinic points}
\label{sec:summable}

Let $\al$ be an algebraic $\zd$-action on a compact abelian group
$X$. We recall from \cite{LS} that a point $x\in X$ is
\emph{homoclinic for $\al$} if $\al^{\bn}(x)\to0$ as
$\|\bn\|\to\infty$. The set of all homoclinic points in $X$ is a
subgroup that we denote by $\Da(X)$. If $\al$ is expansive then the
homoclinic group $\Da(X)$ is countable, and additionally
$\al^{\bn}(x)\to0$ exponentially fast as $\|\bn\|\to\infty$ for every
$x\in\Da(X)$. If $\al$ is nonexpansive then $\Da(X)$ may be uncountable,
countable, or trivial, and $\al$-homoclinic points may decay very slowly
(see \cite{LS} for details, examples, and connections with entropy).

For the proof of Theorem \ref{th:main} we need homoclinic points which
decay sufficiently rapidly. To describe our requirements more precisely,
we write $\dT t \dT$ for the distance from a point $t\in\TT$ to $0$.
A point $x\in X_f$ is called a \emph{summable homoclinic point
for $\af$} if $\sum_{\bn\in\zd}\dT x_{\bn} \dT<\infty$. Denote the
subgroup of summable homoclinic points by $\Dafs(X_f)$, which is
obviously a subgroup of $\Daf(X_f)$.

If $\af$ is expansive then $\U(f)=\emptyset$, and so $1/f^*$ is analytic
on $\sd$. As shown in \cite{LS}, the Fourier coefficients of $1/f^*$
provide a nonzero homoclinic point for $\af$ that decays exponentially
fast, hence is summable. When $\U(f)\ne\emptyset$, the same approach
will work provided there is a $g\in\rd\smallsetminus(f)$ for which
$g/f^*$ is smooth enough to have absolutely convergent Fourier
series. The existence of such a $g$ exactly depends on whether or not
$f$ is atoral.

We begin with some Fourier machinery. For $a=(a_{\bn})\in\lozdc$ define
its Fourier transform $\ahat\colon\sd\to\CC$ by
$\ahat(\bs)=\sum_{\bn\in\zd}a_{\bn}\bs^{\bn}$, where
$\bs^{\bn}=s_1^{n_1}\cdots s_d^{n_d}$. If $\phi\colon\sd\to\CC$ is
integrable with respect to Haar measure $\lam$ on $\sd$, we define its
Fourier coefficients by $\widehat{\phi}_{\bn}=\int
\phi(\bs)\bs^{-\bn}d\lam(\bs)$. A polynomial $g=\sum
g_{\bn}\bu^{\bn}\in\rd$ can be considered as an element in $\lozdc$, and
as such $\ghat$ corresponds to the polynomial function on $\sd$.

We return to our given nonzero irreducible polynomial
$f\in\rd\subset\lozdc$. Define ideals $\nf$ and $\mf$ of $\rd$ by
\begin{displaymath}
    \nf=\{h\in\rd:h|_{\U(f)}\equiv0\},
\end{displaymath}
\begin{displaymath}
    \mf=\{h\in\rd:\hhat/\fhat \text{ has absolutely convergent Fourier
   series}\}.
\end{displaymath}
Clearly $(f)\subseteq\mf\subseteq\nf$. By definition, $f$ is toral if
and only if these ideals coincide.

\begin{lemma}
   \label{lem:radical}
   Let $\sqrt{\mathstrut{}\mf}$ denote the radical ideal of $\mf$. Then
   $\sqrt{\mathstrut\mf}=\nf$. 
\end{lemma}

\begin{proof}
   If $g\in\sqrt{\mathstrut\mf}$, then $g$ (or, more precisely, the
   Fourier transform $\ghat\colon\sd\to\CC$) must vanish on $\U(f)$.
   Hence $\sqrt{\mathstrut\mf}\subseteq\nf$.

   For the reverse inclusion we use the {\L}ojasiewicz inequality from
   real algebraic geometry. Recall that we can consider $\CC^d=\RR^{2d}$
   with coordinates $x_1,y_1,\dots,x_d,y_d$. We may assume that
   $f\in\rd$ is a polynomial, and expand $f=f_1+if_2$ as in
   \eqref{eqn:real-expansion}, where $f_1,f_2\in\RR[x_1,y_1,\dots,
   x_d,y_d]$. Let $F=f_1^2+f_2^2$. The zero set of $F$ in $\sd$ is just
   $\U(f)$. The classical \L ojasiewicz inequality applied to $F$
   implies there are constants $C,\beta$ such that
   \begin{displaymath}
      |\fhat(\bs)|^2=|F(\bs)|\ge C\dist(\bs,\U(f))^{\beta},
   \end{displaymath}
   where $\dist$ denotes the usual distance between points of $\sd$.
   Since $\ghat$ vanishes on $\U(f)$ and is Lipschitz, it follow that
   there is a $k\ge1$ such that the function $G_k\colon\sd\to\CC$
   defined by
   \begin{displaymath}
      G_k(\bs)=
      \begin{cases}
         \displaystyle\frac{\ghat(\bs)^k}{\fhat(\bs)}
                            &\text{if $\bs\notin\U(f)$},\\
         0                  &\text{if $\bs\in\U(f)$}
      \end{cases}
   \end{displaymath}
   is continuous on $\sd$. The standard formula for derivatives of
   quotients then shows that we can arrange for $G_K$ to have as many
   partial derivatives as we need by making $K$ large enough. Since
   sufficiently smooth functions have absolutely convergent Fourier
   series, $g^K\in\mf$ for large enough $K$, and so
   $g\in\sqrt{\mathstrut\mf}$.
\end{proof}

We now state the main result of this section.

\begin{theorem}
   \label{thm:summable-points}
   Let $f\in\rd$ be a nonzero irreducible Laurent polynomial. Let $\af$
   be the algebraic $\zd$-action on $X_f$ defined above, and
   $\Dafs(X_f)$ be the subgroup of summable homoclinic points of
   $\af$. Then the following are equivalent:
   \begin{enumerate}
     \item[(1)]  $\Dafs(X_f)\ne\{0\}$;
     \item[(2)]  $\Dafs(X_f)$ is dense in $X_f$;
     \item[(3)]  $f$ is atoral \textup{(}or, equivalently,
      $(f)\subsetneq\nf$\textup{)};
     \item[(4)]  $\dim\U(f)\le d-2$.
   \end{enumerate}
\end{theorem}

\begin{proof}
   The case $d=1$ is easily handled using the observations that
   $f\in R_1$ is atoral if and only if $\U(f)=\emptyset$, and this occurs
   if and only if $\af$ has nonzero homoclinic points which decay
   exponentially fast, in which case all homoclinic points are summable
   and the homoclinic group is dense (see \cite{LS} for details).

   We may therefore assume that $d\ge2$. The equivalence of (3) and (4)
   is contained in Proposition \ref{prop:atoral}.
   Clearly (2) implies (1) since $f$ is not a unit and so $X_f\ne\{0\}$.

   To prove the remaining implications, we first linearize the action
   $\af$. Consider the surjective map $\eta\colon \lizdr\to\tzd$ given
   by $\eta(w)_{\bn}=w_{\bn}\text{(mod 1)}$.We define the covering
   shift-action $\sigt$ of $\zd$ on $\lizdr$ by
   $(\sigt^{\bm}w)_{\bn}=w_{\bm+\bn}$. Set
   \begin{displaymath}
      f(\sigt)=\sum_{\bn} f_{\bn}\sigt^{\bn}\colon\lizdr\to\lizdr,
   \end{displaymath}
   and put
   \begin{align*}
      W_f\coloneqq \eta^{-1}(X_f) &=\{w\in\lizdr:\eta(w)\in X_f\}\\
      &=\{w\in\lizdr:f(\sigt)(w)\in\lizdz\}.
   \end{align*}
   We view $W_f$ as the \emph{linearization} of $X_f$. Also, viewing
   $f^*$ as an element of $\lozdr$, the point $f(\sigt)(w)$ is the
   convolution product $f^* *w\in\lizdr$.

   Now suppose (3) holds, so that $f$ is atoral. Hence there is a
   $g\in\rd\smallsetminus(f)$ with $\U(g)\supseteq\U(f)=\U(f^*)$. By
   Lemma \ref{lem:radical}, there is a $k$ for which
   $\ghat^k/\fhat\colon\sd\to\CC$ has absolutely convergent Fourier
   transform $v=(v_{\bn})\in\lozdc$. Taking the Fourier transform of
   $\ghat^k/\fhat=\widehat{v}$ shows that $f*v=g^k\in\lizdz$, and so $f^*
   *v^*=(g^*)^k$ has integral coordinates, so that $v^*\in W_f$. Hence
   the images under $\eta$ of the real and imaginary parts of $v^*$ give
   nonzero points in $\Dafs(X_f)$, proving (1).

   Finally, suppose (1) holds, and let $0\ne x\in\Dafs(X_f)$. Lift $x$
   to $v\in W_f\cap\lozdr$ with $\eta(v)=x$. Since $v_{\bn}\to0$ as
   $\|\bn\|\to\infty$, it follows that $f^* *v\in\lizdz$ can have only
   finitely many nonzero coordinates, so represents a Laurent polynomial, say
   $h\in\rd$. Thus $\hhat=\fhs\cdot\widehat{v}$. This shows that
   $\hhat$ vanishes on $\U(f^*)=\U(f)$, and so $h^*\in\nf$. If
   $h^*=g\cdot f$ for some $g\in\rd$, we would have $h=g^*\cdot f^*$, and so
   $\widehat{v}\fhs=\ghs\cdot\fhs$, and a continuity argument
   shows that $v=g^*\in\lizdz$. But then $x=\eta(v)=0$, a
   contradiction. Thus $h^*\in\nf\smallsetminus(f)$, and so $f$ is
   atoral, proving (3). Furthermore $h$ is relatively prime to $f$, so
   that multiplication by $h$ is injective on $\rd/(f)$. It follows that
   $h(\af)\colon X_f\to X_f$ is surjective. The proof of \cite[Lemma~
   6.3]{LSV} shows that the subgroup of $\Dafs(X_f)$ generated by the
   shifts of $x$ is dense in $h(\af)(X_f)=X_f$, proving (2) and
   completing the proof.
\end{proof}

We extract one consequence of the preceding proof as a corollary.

\begin{corollary}
   Let $f\in\rd$ be an irreducible atoral polynomial, and let $\af$ be
   the algebraic $\zd$-action on $X_f$ as described above. Then every
   $x\in\Dafs(X_f)$ has the form $x_h=\eta(v_h)$, where $h^*\in\mf$ and
   $v_h\in\lozdr$ is the Fourier transform of
   $\hhat/\fhs$. Furthermore, $x_h$ is nonzero if and only if
   $h^*\in\mf\smallsetminus(f)$, so that $\Dafs(X_f)$ is isomorphic as
   an abelian group to $\mf/(f)$.
\end{corollary}

In \cite[Thm.~4.2]{LS} it is shown that, for expansive algebraic
$\zd$-actions, the homoclinic group is dense if and only if the action
has completely positive entropy. Using Theorem
\ref{thm:summable-points}, these properties also hold for $\af$ when $f$
is atoral.

\begin{corollary}
   \label{cor:density}
   Let $f\in\rd$ be a \textup{(}possibly reducible\textup{)} atoral
   Laurent polynomial that is not a unit in $\rd$. Then $\af$ has
   completely positive entropy and $\Dafs(X_f)$ is dense in $X_f$.
\end{corollary}

\begin{proof}
   First recall from \cite[Thm.~1]{Boyd} that an irreducible Laurent
   polynomial $g\in\rd$ has $\h(\alpha_g)=0$ if and only if there is a
   one-variable cyclotomic polynomial $c(t)$ and $\bm,\bn\in\zd$ such
   that $g(\bu)=\pm\bn^{\bm}c(\bu^{\bn})$. For such $g$ clearly
   $\dim\U(g)=d-1$, and so $g$ is toral.

   Let $f=f_1^{k_1}\cdots f_r^{k_r}$ be the factorization of $f$ into
   irreducible Laurent polynomials. The prime ideals associated with
   $(f)$ are just $(f_1),\dots,(f_r)$, and $\h(\alpha_{R_d/(f_j)})>0$
   for $1\le j\le r$ since each $f_j$ is atoral. Hence by
   \cite[Thm.~6.5]{LSW}, it follows that $\af$ has completely positive
   entropy.

   Suppose that $g$ and $h$ are relatively prime Laurent
   polynomials. Then
   \begin{equation}
      \label{eq:sums}
      X_{gh}=(gh)^{\perp}=[(g)\cap (h)]^{\perp}=(g)^{\perp}+(h)^{\perp}=X_g+X_h.
   \end{equation}
   
   Now by definition each $f_j$ is atoral, and so by the proof of
   Theorem~\ref{thm:summable-points} there is a $g_j$ for which
   $g_j/f_j$ has absolutely convergent Fourier series. Hence
   $g_j^{k_j}/f_j^{k_j}$ also has absolutely convergent Fourier series.
   If we denote
   $\alpha_{f_j^{k_j}}$ by $\alpha_j$ and $X_{f_j^{k_j}}$
   by $X_j$, the proof of Theorem ~\ref{thm:summable-points} shows that
   $\Delta_{\al_j}^1(X_j)$ is dense in $X_j$ for $1\le j\le r$. Then
   $\Delta_{\al_1}^1(X_1)+\dots+\Delta_{\al_r}^1(X_r)$ is contained in
   $\Dafs(X_f)$, and is dense in $X_1+\dots+X_r$, which is equal to $X$
   by \eqref{eq:sums}.
\end{proof}

\section{Examples}
\label{sec:examples}

We give here examples to illustrate various phenomena. For clarity,
we use variables $u,v,w$ instead of $u_1,u_2,u_3$.

\begin{example}
   \label{exam:hyperbolic-toral}
   Let $d=1$ and $f(u)=u^2-u-1$. Then $f(u)$ has roots
   $\lam=(1+\sqrt{5})/2$ and $\mu=(1-\sqrt{5})/2$, so that
   $\U(f)=\emptyset$ and $f$ is atoral. As discussed in \cite[Example
   6.7]{LS}, we can compute the Fourier series of
   $1/f^*(u)=1/(u^{-1}-\lam)(u^{-1}-\mu)$ by partial fractions and
   obtain the coefficients
   \begin{displaymath}
      w_n^{\Delta}=
      \begin{cases}
         \displaystyle -\frac{1}{\sqrt5}\,\mu^{n-1} &\text{if
         $n\ge1$},\\
         \displaystyle -\frac{1}{\sqrt5}\,\lam^{n-1} &\text{if
         $n\le0$}.
      \end{cases}
   \end{displaymath}
   Then $x^{\Delta}=\eta(w^{\Delta})\in\Dafs(X_f)$, and
   $\Daf(X_f)=\Dafs(X_f)$ is the group generated by all translates of
   $x^{\Delta}$.

   Geometrically, $\af$ is a hyperbolic automorphism of the
   2-dimensional torus $X_f$, and $\Dafs(X_f)$ is the dense intersection
   of the 1-dimensional stable and unstable eigenlines of $\af$.
\end{example}

\begin{example}
   \label{exam:quasi-hyperbolic}
   Let $d=1$ and $f(u)=u^4-u^3-u^2-u+1$. Then
   $\U(f)=\{\xi,\overline{\xi}\}$, where
   \begin{displaymath}
      \xi=\displaystyle \frac{1-\sqrt{13}}{3}+\frac14\sqrt{2+2\sqrt{13}}
      \,\,i \in\SS,
   \end{displaymath}
   so that $f$ is toral. This can be seen directly since $f$ is the
   minimal polynomial of $\xi$ over $\QQ$, so that any $g\in R_1$ with
   $\U(g)\supseteq\U(f)$ must be in $(f)$. As shown in \cite[Example
   3.4]{LS}, $\Daf(X_f)=\{0\}$, because the 1-dimensional stable and
   unstable eigenlines have only trivial intersection. Here $(f)=\nf=\mf$.
\end{example}

\begin{example}
   \label{exam:harmonic}
   Let $d=2$ and $f(u,v)=2-u-v$. Then $\U(f)=\{(1,1)\}$ and so $f$ is
   atoral. For example, $g(u,v)=u-1$ has $\U(g)\supseteq\U(f)$, but
   $g\notin(f)$.

   As discussed in \cite[Sect.\ 5]{LSV}, $1/f^*$ is integrable on
   $\SS^2$, with Fourier coefficients
   \begin{displaymath}
      w_{(-m,-n)}^{\Delta}=
      \begin{cases}
         \displaystyle \frac{1}{2^{m+n+1}}\binom{m+n}{n} &\text{if
         $m\ge0$ and $n\ge0$,}\\
         0 &\text{otherwise}.
      \end{cases}
   \end{displaymath}
   Hence $x^{\Delta}=\eta(w^{\Delta})\in\Daf(X_f)$, but $x^{\Delta}$ is
   \emph{not} summable since, for example,
   \begin{displaymath}
      x_{(-n,-n)}^{\Delta}=\displaystyle
      \frac{1}{2^{2n+1}}\binom{2n}{n}\approx \frac{1}{2\sqrt{\pi n}}
   \end{displaymath}
   decays too slowly.

   We can attempt to speed up the rate of decay by applying difference
   operators, or equivalently by multiplying $1/f^*$ by powers of $u-1$
   and $v-1$. It turns out that third powers are exactly what is needed,
   so that for example $(u-1)^3/f^*(u,v)$ has absolutely convergent
   Fourier series whose coefficients provide a nonzero point in
   $\Dafs(X_f)$.

   Here $\nf=\{g\in R_2:g(1,1)=\sum_{\bn}g_{\bn}=0\}$ and $\mf$ is the
   ideal in $R_2$ generated by $(u-1)^p(v-1)^q$ with $p+q=3$. The
   ``summable defect'' in this example is
   \begin{displaymath}
      \Daf(X_f)/\Dafs(X_f)\cong \nf/\mf,
   \end{displaymath}
   which is a finitely generated abelian group of rank five generated by
   the cosets $(u-1)^r(v-1)^s+\mf$ with $1\le r+s\le2$.
\end{example}

\begin{example}
   \label{exam:subharmonic}
   Let $d=2$ and $f(u,v)=3-u-u^{-1}-v-v^{-1}$. Then $\U(f)$ is a
   1-dimensional curve in $\SS^2$ which is the image under the
   exponential map $e\colon \TT^2\to\SS^2$ of the closed curve given by
   \begin{displaymath}
      t = \pm \frac1{2\pi}\cos^{-1}\Bigl(\frac32 -\cos 2\pi s\Bigr),
      \text{\quad $-\frac16\le s\le \frac16$}.
   \end{displaymath}
   Thus $f$ is toral, and so $\Dafs(X_f)=\{0\}$. It follows easily from
   \cite{Mann} that all elements of $\U(f)\cap\Om$ must have order 30. A
   straightforward search verifies that
   $\U(f)\cap\Om=\{(\om,1),(\overline{\om},1),(1,\om),(1,\overline{\om})\}$,
   where $\om=e^{2\pi i/6}$. Hence $\dim \FGo(\af)\le4$ for all $\G\in\SF$.

   Ironically, $\Daf(X_f)$
   is so large that it is uncountable. Let $\mu$ be any measure
   supported on $\U(f)$ that is a smooth function multiple of arc length. Since
   the curvature of $\U(f)$ is bounded away from zero, a result of Stein
   \cite[Thm.\ 2 of \S VIII.3.2]{Stein} shows that $\widehat{\mu}(\bn)\to0$ as
   $\|\bn\|\to\infty$. Hence the point $x$ with
   $x_{\bn}=\eta(\Re[\widehat{\mu}(\bn)])$ is in $\Daf(X_f)$. However, the
   same result shows that $\widehat{\mu}(\bn)$ must decay so slowly that
   $x$ is never summable. 

   The dynamical properties of $\af$ are still
   somewhat mysterious. In particular, we do not know whether Theorem
   ~\ref{th:main} holds here.
\end{example}

\begin{example}
   Let $d=2$ and $f(u,v)=u^2+u^{-2}-2u-2u^{-1}+v+v^{-1}$. Then $f$ is
   irreducible, and it is easy to check that $\U(f)$ consists of two
   1-dimensional curves together with the point $(1,1)$. This
   illustrates the possibility that the connected components of $\U(f)$
   may have different dimensions.
\end{example}

\begin{example}
   \label{exam:1+x+y+z}
   Let $d=3$ and $f(u,v,w)=1+u+v+w$. This example has appeared in the
   literature due to the interesting value of its logarithmic Mahler
   measure \cite{Smyth}: $\m(f)=7\zeta(3)/2\pi^2$. It is easy to verify
   that $\U(f)$ is the union of three circles $\{(-1,s,-s):s\in\SS\}$,
   $\{(s,-1,-s):s\in\SS\}$, and $\{(s,-s,-1):s\in\SS\}$. Hence $f$ is
   atoral. For example, both $uvw+uv+uw+vw$ and $(u-v)(v-w)(w-u)$ vanish
   on $\U(f)$ but are not in $(f)$.

   Observe that $\U(f)$ contains infinitely many torsion points. Indeed,
   if $\G_n=n\ZZ^3$, then
   \begin{displaymath}
      |\Omega_{\G_n}\cap\U(f)|=
      \begin{cases}
         3n-3 &\text{if $n$ is even}\\
         0    &\text{if $n$ is odd}. 
      \end{cases}
   \end{displaymath}
   Hence the dimension of $\Fix_{\G_n}^\circ(\af)$ is unbounded, an
   issue that will need to be dealt with in the proof of our main
   theorem (see Lemma \ref{lem:mann}).

   This example is a special case of the fact that if $\U(f)$ has
   infinitely many torsion points, then all of these must lie on a
   finite union of cosets of rational subtori.
\end{example}

\begin{example}
   \label{exam:2+x+y+z}
   Let $d=3$ and $f(u,v,w)=2+u+v+w$, which is irreducible. Since $f$ is
   not symmetric, it is atoral. One can consider points in $\U(f)$ as
   possible positions of a closed linkage of four rods, one fixed horizontal
   rod of length 2, and three others of length 1 joined end to end. This
   system has one degree of freedom, and so $\U(f)$ is a smooth loop in
   $\SS^3$.

   By Theorem \ref{thm:summable-points}, $\Dafs(X_f)\ne\{0\}$ (in fact is
   dense in $X_f$). Now $\U(f)$ is a smooth curve with finite order of
   contact with every hyperplane (i.e., has \emph{finite type} in the
   terminology of \cite{Stein}).  Hence if $\mu$ is any measure on $\U(f)$
   that is a smooth multiple of arc length, then by \cite[Thm.\ 2, \S
   VIII.3.2]{Stein} we have that $\widehat{\mu}(\bn)\to0$ as
   $\|\bn\|\to\infty$. Hence $x=(x_{\bn})$, where
   $x_{\bn}=\eta(\Re[\widehat{\mu}(\bn)])$, is in $\Daf(X_f)$, but decays too
   slowly to be summable.

   In this example  $\Dafs(X_f)$ is
   countable and dense in $X_f$ by Corollary \ref{cor:density}, while
   $\Daf(X_f)\smallsetminus\Dafs(X_f)$ is uncountable.
\end{example}

\begin{example}
   \label{exam:complex-factorization}
   The roots of $u^2-9u+19$ are $\xi\approx 3.3819$ and $\zeta\approx
   5.6180$. Consider
   \begin{align*}
      f(u,v)&\coloneqq g_1(u,v)g_2(u,v)=(\xi-u-u^{-1}-v-v^{-1})
      (\zeta- u-u^{-1}-v-v^{-1})\\
      &=23
      +u^2+u^{-2}-9u-9u^{-1}+v^2+v^{-2}-9v-9v^{-1}\\
      &\quad\quad\quad\quad\quad+ 2uv+2u^{-1}v^{-1}+2uv^{-1}+2u^{-1}v.  
   \end{align*}
   Here $f$ is irreducible in $R_2$, or equivalently in
   $\QQ[u^{\pm1},v^{\pm1}]$, but factors in
   $\CC[u^{\pm1},v^{\pm1}]$. Furthermore, $\U(g_1)$ is a 1-dimensional
   curve in $\SS^2$, analogous to Example \ref{exam:subharmonic}, while
   $\U(g_2)=\emptyset$, and so $f$ is toral in our sense. However, $f$
   is neither toral nor atoral in the sense of \cite{AMS} since it has
   mixed factors over $\CC$.
\end{example}

\begin{remark}
   If $f\in\rd$ is irreducible, but factors over $\CC$, it can be shown
   that except for a trivial scalar normalization, each factor has
   coefficients that are algebraic numbers, i.e. this factorization
   already takes place in $\overline{\QQ}[u_1^{\pm1},\dots,u_d^{\pm1}]$.
   Indeed, by Dedekind's Prague Theorem \cite[p.\ 2]{Edwards}, the
   coefficients of each factor can be taken to be algebraic integers.
\end{remark}

\section{Symbolic covers and specification}
\label{sec:symbolic-covers}

Let $f\in\rd$ be an irreducible atoral Laurent polynomial and $\af$ be
the corresponding algebraic $\zd$-action on $X_f$. By Theorem
\ref{thm:summable-points}, there are nonzero summable homoclinic points
for $\af$. Fix one of these, say $x$. As pointed out in \cite{LS}, there
is a surjective, shift-equivariant map $\xi_x\colon\lizdz\to X_f$
defined by
\begin{displaymath}
   \xi_x(v)=\sum_{\bn\in\zd}v_{\bn}\af^{-\bn}(x),
\end{displaymath}
where coordinatewise convergence follows from summability of $x$. In
fact, $\xi_x$ is surjective when restricted to a suitably large ball of
radius $K$ in $\lizdz$, thereby providing a symbolic cover for $X_f$
with symbols $\{-K,\dots,K\}$.

Having established the existence of summable homoclinic points for
atoral polynomials here, the proof in \cite[Thm.\ 8.2]{LSV} using
symbolic covers applies to yield the following remarkably strong
specification properties of $\af$.

\begin{proposition}[{\cite[Thm.\ 8.2]{LSV}}]
   \label{thm:specification}
   Let $f\in\rd$ be an irreducible atoral polynomial
   and $\af$ be the corresponding algebraic $\zd$-action on $X_f$. Fix a
   translation-invariant metric $\delta$ on $X_f$. Then for every
   $\eps>0$ there exists a number $p(\eps)>0$ with the following
   properties:
   \begin{enumerate}
     \item For every finite collection $\{Q_1,\dots,Q_r\}$ of finite
      subsets of $\zd$ with
      \begin{displaymath}
         \textup{(*)}\quad
         \textup{dist}(Q_j,Q_k)\coloneqq \min_{\bm\in Q_j,\ \bn\in Q_k}
         \|\bm-\bn\|\ge p(\eps) \enspace \text{for $1\le j<k\le r$},
      \end{displaymath}
      every collection $\{x^{(1)},\dots,x^{(r)}\}\subset X_f$, and every
      $\G\in\SF$ with
      \begin{displaymath}
         \textup{dist}(Q_j+\bk,Q_k)\ge p(\eps) \text{ for $1\le j<k\le
          r$ and every $\bk\in\G\smallsetminus\{\bzero\}$,}
       \end{displaymath}
       there is a $y\in\FG(\af)$ with
       \begin{displaymath}
          \textup{(**)}\quad\quad\ \ 
          \delta\bigl(\af^{\bn}(y),\af^{\bn}(x^{(j)})\bigr)<\eps \text{ for $1\le
          j\le r$ and every $\bn\in Q_j$}.
       \end{displaymath}
      \item For every finite  collection $\{Q_1,\dots,Q_r\}$ of finite
      subsets of $\zd$ satisfying \textup{(*)} and  every
      collection $\{x^{(1)},\dots,x^{(r)}\}\subset X_f$ there is a point
      $y\in\Dafs(X_f)$ satisfying \textup{(**)}.
   \end{enumerate}
\end{proposition}

\section{Proof of the main theorem}
\label{sec:main}

To begin the proof of Theorem \ref{th:main}, we fix an atoral
irreducible Laurent polynomial $f\in\rd$, and let $\af$ be the
corresponding cyclic algebraic $\zd$-action on $X_f$.

Roughly speaking, with the availability of summable homoclinic points,
obtaining sufficiently many separated points in $\FG(\af)$ is relatively
easy. However, many of these could lie in the same coset of the
connected component $\FGo(\af)$ of the identity. In order to show that
this does not affect the logarithmic growth rate of $\PG(\af)$, we
invoke a result of H.\ B.\ Mann which shows that $\U(f)\cap\Om$ lies in
the union of a finite number of cosets of rational subtori of
$\sd$. This enables us to embed $\FGo(\af)$ in a finite sum of subtori
with spanning sets whose cardinality have logarithmic growth rate
zero. This will force the number of cosets of $\FGo(\af)$ in $\FG(\af)$
to have the correct logarithmic growth rate.

We begin with some terminology.

\begin{definition}
   Let $Q\subset\zd$ be a finite set. Define a pseudometric $\dq$ on
   $\tzd$ by setting
   \begin{displaymath}
      \dq(x,y)=\max_{\bn\in Q} \dT x_{\bn}-y_{\bn}\dT,\text{\quad for
      $x,y\in\tzd$},
   \end{displaymath}
   where $\dT x-y \dT$ is the usual distance between $s,t\in\TT$.

   A set $F\subset\tzd$ is \emph{$(Q,\eps)$-separated} if
   $\dq(x,y)>\eps$ for every pair $x,y$ of distinct points in $F$. If
   $Y\subset\tzd$, then a set $F\subset Y$ is \emph{$(Q,\eps)$-spanning}
   if, for every $y\in Y$, there is an $x\in F$ with $\dq(x,y)<\eps$.

   If $\G\in\SF$, we write $\dg$ for the metric on $\FG(\sigma)$ defined
   by $\dg=\dq$ for any fundamental domain $Q$ for $\G$. A set
   $F\subset\FG(\sigma)$ is \emph{$(\G,\eps)$-separated} if
   $\dg(x,y)>\eps$ for every pair $x,y$ of distinct points in $F$, and
   there is an analogous definition of \emph{$(\G,\eps)$-spanning}.
\end{definition}

\begin{lemma}
   \label{lem:shrink}
   For every $\eps>0$ there is a finite set $A_{\eps}\subset\zd$ such
   that, for every $\G\in\SF$ and every fundamental domain $Q\subset\zd$
   of $\G$, the set $\FG(\af)$ is $\bigl(\bigcap_{\bm\in
   A_{\eps}}(Q-\bm)\bigr)$-spanning in $X_f$.
\end{lemma}

\begin{proof}
   Since $f$ is atoral, our discussion in Section \ref{sec:summable}
   shows that there is a $g\in\rd$ relatively prime to $f$ such that the
   Fourier transform of $g/f$ provides a summable homoclinic point for
   $\af$. The proof of Lemma 7.3 in \cite{LSV} is then also valid in
   this situation. 
\end{proof}

The next lemma is an easily proved special case of \cite[Cor.\ 5.6]{DS}.

\begin{lemma}
   \label{lem:folner}
   Let $(\G_n)_{n\ge1}$ be a sequence in $\SF$ with $\<\G_n\>\to\infty$
   as $n\to\infty$. Then there exists a sequence $(Q_n)_{n\ge1}$ of
   finite subsets of $\zd$ such that each $Q_n$ is a fundamental domain
   for $\G_n$, and also that $(Q_n)_{n\ge1}$ is a F{\o}lner sequence for
   $\zd$. 
\end{lemma}

\begin{lemma}
   \label{lem:enough-points}
   Let $\eps>0$ and let $(\G_n)_{n\ge1}$ be a sequence in $\SF$ with
   $\<\G_n\>\to\infty$. Choose a F{\o}lner sequence $(Q_n)_{n\ge1}$ of
   fundamental domains for the groups $(\G_n)_{n\ge1}$ as in Lemma
   \ref{lem:folner}. Let $A_{\eps}$ be the finite set chosen according
   to Lemma \ref{lem:shrink}. Define $$Q_n'=\bigcap_{\bm\in
   A_{\eps}}(Q_{n}-\bm).$$

   Then $(Q_n')_{n\ge1}$ is again a F{\o}lner sequence, and
   $|Q_n'|/|Q_n|\to1$ as $n\to\infty$. Furthermore, for every $n\ge1$
   there exists a $(\G_n,\eps)$-separated set
   $F_n(\eps)\subset\FGn(\af)$ which is $(Q_n',2\eps)$-spanning in $X_f$.
\end{lemma}

\begin{proof}
   From Lemma \ref{lem:shrink} we know that $\FGn(\af)$ is
   $(Q_n',\eps)$-spanning in $X_f$ for all $n\ge1$. Let
   $F_n(\eps)\subset\FGn(\af)$ be a $(\G_,\eps)$-separated set of
   maximal cardinality. Then $F_n(\eps)$ is also $(\G_n,\eps)$-spanning
   in $\FGn(\af)$, hence $(Q_n',2\eps)$-spanning in $X_f$.
\end{proof}

It follows from \cite[Prop.\ 2]{D} that 
\begin{equation}
   \label{eqn:liminf}
   \lim_{\eps\to0}\,
   \liminf_{n\to\infty}\frac1{|\zd/\G_n|}\,\log|F_n(\eps)|=\h(\af). 
\end{equation}
The next, and more difficult, step in our proof consists of showing that
the number of distinct cosets of $\FGno(\af)$ intersecting $F_n(\eps)$
nontrivially has the same logarithmic growth rate as $F_n(\eps)$.

We first introduce some notation to linearize our situation. We write
$\lizdc$ for the space of bounded complex-valued functions on $\zd$, and
$\sigt$ for the $\zd$-shift action on this space. Similarly, $\sigt$ acts on the real
part $\lizdr$ of $\lizdc$. For $\G\in\SF$ we let $\lzdgc$ and $\lzdgr$
be the corresponding finite-dimensional $\G$-periodic subspaces.
For each $\bo\in\OG$ there is an element $\vo$ in $\lzdgc$ defined by
$\vo_{\bn}=\bo^{\bn}$ for all $\bn\in\zd$.

A set $S\subset\OG$ is called \emph{symmetric} if it is closed under
taking inverses. A function $c\colon S\to\CC$ on a symmetric set is
\emph{skew-symmetric} if $c(\bo^{-1})=\overline{c(\bo)}$. Let
$V(S,\CC)\subset\lzdgc$ denote the complex span of the points $\vo$
where $\bo\in S$, and $V(S,\RR)=V(S,\CC)\cap\lzdgr$. Then $V(S,\RR)$
consists of all sums of the form $\sum_{\bo\in S}c(\bo)\vo$ where $c$ is
skew-symmetric on $S$, which has real dimension $|S|$.

For $\Delta\in\SF$ let $B_1(\lzddr)$ denote the unit ball in $\lzddr$
with respect to the $\ell^{\infty}$-norm.

\begin{lemma}
   \label{lem:estimate}
   Let $\Delta\in\SF$ and $0<\eps<1$. Then there is a
   $(\Delta,\eps)$-spanning set $F\subset B_1(\lzddr)$ with cardinality
   $|F|<(2/\eps)^{|\zd/\Delta|}$. Hence
   $\eta(\lzddr)=\Fix_{\Delta}(\sig)=\{x\in\tzd:\sig^{\bn}x=x\text{ for
   all $\bn\in\Delta$}\}$ has a $(\Delta,\eps)$-spanning set of
   cardinality $<(2/\eps)^{|\zd/\Delta|} $.
\end{lemma}

\begin{proof}
   Let $Q$ be a fundamental domain for $\Del$. Then
   $\lzddr\cong\ell^{\infty}(Q,\RR)$ with the $\ell^{\infty}$-norm. For
   each $\bq\in Q$ let $F_{\bq}=\{j\eps 1_{\{\bq\}}:-1/\eps<j<1/\eps\}$,
   so $|F_{\bq}|<2/\eps$. Put $F=\sum_{\bq\in Q}F_{\bq}$. Clearly $F$ is
   $(\Del,\eps)$-spanning for $B_1(\lzddr)$, and
   $|F|<(2/\eps)^{|Q|}=(2/\eps)^{|\zd/\Del|}$.

   Finally, $\eta(B_1(\lzddr))=\Fix_{\Del}(\sig)$ and $\eta$ is a local
   isometry. Hence $\eta(F)$ is a $(\Del,\eps)$-spanning set for
   $\Fix_{\Delta}(\sig)$ of cardinality $<(2/\eps)^{|\zd/\Del|}$.
\end{proof}

We will be using proper closed subgroups of $\sd$ to capture the torsion
points in $\U(f)$, which are responsible for the dimension of
$\FGo(\af)$. To do so, we define, for every $\bzero\ne\bm\in\zd$ the
subgroup $\Hm=\{\bs\in\sd:\bs^{\bm}=1\}\subset\sd$. Observe that
$\Hm^{\perp}\coloneqq\{\bn\in\zd:\bs^{\bn}=1\text{ for all
$\bs\in\Hm$}\}$ is just $\ZZ\bm$. The next lemma allows us to estimate
the size of the slice of $\OG$ contained in $\Hm$.

\begin{lemma}
   \label{lem:slice}
   Let $\bzero\ne\bm\in\zd$ and $\G\in\SF$. Then
   \begin{displaymath}
      |\OG\cap\Hm|\le\frac{\|\bm\|}{\<\G\>}\,|\OG|.
   \end{displaymath}
\end{lemma}

\begin{proof}
   Basic duality shows that $\OG/(\OG\cap\Hm)$ is isomorphic with
   $(\OG\cap\Hm)^{\perp}/\OG^{\perp}$, and that $\OG^{\perp}=\G$ and
   $(\OG\cap\Hm)^{\perp}=\OG^{\perp}+\Hm^{\perp}=\G+\ZZ\bm$. Hence
   $|\OG/(\OG\cap\Hm)|=|(\G+\ZZ\bm)/\G|$, which is just the order $k$ of
   $\bm$ in $\zd/\G$. Since $\bzero\ne k\bm\in\G$, we have
   $\|k\bm\|\ge\<\G\>$, so that $k\ge\<\G\>/\|\bm\|$. Hence
   \begin{displaymath}
      |\OG\cap\Hm|=\frac1k\,|\OG|\le\frac{\<\G\>}{\|\bm\|}\,|\OG|.
      \qedhere
   \end{displaymath}
\end{proof}

\begin{lemma}
   \label{lem:coset}
   Let $\bzero\ne\bm\in\zd$ and $\bs\in\sd$. If
   $(\bs\cdot\Hm)\cap\Om\ne\emptyset$, then there is a $k\ge1$ such that
   $\bs\cdot\Hm\subset H_{k\bm}$.
\end{lemma}

\begin{proof}
   If $\bs\cdot\Hm\cap\Om\ne\emptyset$, then there is an $\bs'\in\Om$
   such that $\bs'\cdot\Hm=\bs\cdot\Hm$. Let $k$ be the order of $\bs'$
   in $\sd$.
\end{proof}

\begin{lemma}
   \label{lem:mann}
   Let $0\ne f\in\rd$. Then there are nonzero $\bm_1,\dots,\bm_L\in\zd$
   such that
   \begin{displaymath}
      \U(f)\cap\Om\subset\bigcup_{j=1}^L H_{\bm_j}.
   \end{displaymath}
\end{lemma}

\begin{proof}
   A special case of the Manin-Mumford conjecture established by Mann
   \cite{Mann} shows that $\U(f)\cap\Om$ is contained in a finite union
   of cosets of connected proper rational subtori (i.e., closed,
   connected subgroups) of $\sd$. By Lemma \ref{lem:coset}, we can embed
   these cosets into appropriate $H_{\bm_j}$.
\end{proof}

\begin{remark}
   There is a finite procedure for computing the $\bm_j$ appearing in
   the previous lemma (see \cite[\S3.1]{GR} or \cite[\S6]{AS}).
\end{remark}

\begin{lemma}
   Let $\G\in\SF$ and $0\ne f\in\rd$. Define
   $f(\sigt)\colon\lizdr\to\lizdr$ as above, and put $V_{\G}(f)=\ker
   f(\sigt)\cap\lzdgr$. Then $V_{\G}(f)=V(\U(f)\cap\OG,\RR)$ and
   $\eta(V_{\G}(f))=\FGo(\af)$. 
\end{lemma}

\begin{proof}
   This assertion is contained in the proof of \cite[Lemma 6.8]{LSV},
   but we include it here for the convenience of the reader. Since $f$
   has real coefficients, $\U(f)\cap\OG$ is symmetric. Hence $V\coloneqq
   V(\U(f)\cap\OG,\RR)$ consists of the skew-symmetric combinations
   $\sum_{\bo\in\U(f)\cap\OG}c(\bo)\vo$, so the first statement
   follows. Clearly $V$ is a linear subspace of $\lizdr$ and
   $\eta(V)\subset\FGo(\af)$ by connectedness.

   Let $\eps=\bigl(2\sum_{\bn\in\zd}|f_{\bn}|\bigr)^{-1}$, and consider
   the neighborhood
   $N_{\eps}=\{x\in\FGo(\af):\sup_{\bn}\dT x_{\bn} \dT<\eps\}$. If
   $\widetilde{B}_{\eps}=\{v\in\lzdgr:\|v\|_{\infty}<\eps\}$ and
   $\widetilde{N}_{\eps}=\eta^{-1}(X_f)\cap\widetilde{B}_{\eps}$, then
   $\eta(\widetilde{N}_{\eps})=N_{\eps}$. For every
   $v\in\widetilde{N}_{\eps}$ we have that $\eta(f(\sigt)v)=0$, so that
   $f(\sigt)(v)\in\ell(\zd/\G,\ZZ)$. But our choice of $\eps$ then
   forces $f(\sigt)(v)=0$, so that $v\in V$. Hence $\eta(V)\supset
   N_{\eps}$. Since $\FGo(\af)$ is connected, it follows that
   $\eta(V)\supset\FGo(\af)$, and hence $\eta(V)=\FGo(\af)$.
\end{proof}

\begin{lemma}
   \label{lem:capture}
   Let $0\ne f\in\rd$ and choose nonzero $\bm_1,\dots,\bm_L\in\zd$
   according to Lemma \ref{lem:mann} so that $\U(f)\cap\Omega\subset
   H_{\bm_1}\cup\dots\cup H_{\bm_L}$. Let
   $\DGj=(\OG\cap\Hmj)^{\perp}\supset\G$. Then
   \begin{displaymath}
      \FGo(\af)\subset\sum_{j=1}^L \Fix_{\DGj}(\sig).
   \end{displaymath}
\end{lemma}

\begin{proof}
   Let $S_j=\OG\cap\Hmj$ and $S=S_1\cup\dots S_L$, which is clearly
   symmetric. Then
   \begin{align*}
      \FGo(\af) &=\eta(V(\OG\cap\U(f),\RR) \subset
      \eta(V(S,\RR))=\sum_{j=1}^L\eta(V(S_j,\RR)) \\
      &=\sum_{j=1}^L
      \eta(\ell(\zd/\DGj,\RR))=\sum_{j=1}^L\Fix_{\DGj}(\sig).
      \qedhere
   \end{align*}
\end{proof}

\begin{proof}[Proof of Theorem \ref{th:main}] Let $(\G_n)_{n\ge1}$ be a
   sequence in $\SF$ with $\<\G_n\>\to\infty$, and fix $\eps>0$ for the
   moment. Chose $\bm_1,\dots,\bm_L\in\zd$ for $f$ according to Lemma
   \ref{lem:mann}. Put $M=\max_{1\le j\le L}\|\bm_j\|$. Let
   $\Del_{n,j}=(\Om_{\G_n}\cap\Hmj)^{\perp}\supset\G_n$.
   
   Applying Lemma \ref{lem:estimate} to each $\Del_{n,j}$, we obtain
   $(\Del_{n,j},\eps/2L)$-spanning (and hence $(\G_n,\eps/2L)$-spanning)
   sets $F_{n,j}\subset\Fix_{\Del_{n,j}}(\sig)$ of cardinality
   $|F_{n,j}|<(4L/\eps)^{|\zd/\Del_{n,j}|}$. Let
   $F_n=F_{n,1}+\dots+F_{n,L}$. Clearly $F_n$ is
   $(\G_n,\eps/2)$-spanning for the subgroup
   \begin{displaymath}
      H_n=\sum_{j=1}^L\Fix_{\Del_{n,j}}(\sig)\subset\Fix_{\G_n}(\sig)\subset\tzd,
   \end{displaymath}
   and
   \begin{displaymath}
      |F_n|\le\prod_{j=1}^L|F_{n,j}|\le
      \prod_{j=1}^L\Bigl(\frac{4L}{\eps}\Bigr)^{|\Om_{\G_n}\cap\Hmj|}
      \le\Bigl(\frac{4L}{\eps}\Bigr)^{\frac{LM|\Om_{\G_n}|}{\<\G_n\>}}. 
   \end{displaymath}
   Since $|\zd/\G_n|=|\Om_{\G_n}|$ and $\<\Gn\>\to\infty$, it follows that
   \begin{displaymath}
      \lim_{n\to\infty}\frac1{|\zd/\G_n|}\,\log|F_n|=0.
   \end{displaymath}
   Since the cardinality of every $(\G_n,\eps)$-separated set is bounded
   by the cardinality of every $(\G_n,\eps/2)$-spanning set, every
   sequence of $(\G_n,\eps)$-separated sets $F_n'\subset H_n$ must also
   satisfy that
   \begin{equation}
      \label{eqn:separated}
      \lim_{n\to\infty}\frac1{|\zd/\G_n|}\,\log|F_n'|=0.
   \end{equation}
   Choose the finite set $A_\eps$ according to Lemma
   \ref{lem:shrink}. Let $(Q_n)_{n\ge1}$ be a F{\o}lner sequence of
   fundamental domains for the groups $(\G_n)_{n\ge1}$, and put
   $Q_n'=\bigcap_{\bm\in A_\eps}(Q_n-\bm)$. Then by Lemma
   \ref{lem:folner}, we have that $(Q_n')_{n\ge1}$ is again a F{\o}lner
   sequence with $|Q_n'|/|Q_n^{}|\to1$ as $n\to\infty$. From Lemma
   \ref{lem:enough-points} we know that $\Fix_{\G_n}(\af)$ contains a
   $(\G_n,\eps)$-separated set $F_n(\eps)\subset X_f$ which is
   $(Q_n',2\eps)$-spans $X_f$.

   For $n\ge1$, the intersection of $F_n(\eps)$ with every coset of the
   group $H_n$ in $\Fix_{\G_n}(\sig)$ is $(\G_n,\eps)$-separated. We set
   \begin{displaymath}
      D_n(\eps)=\max_{y\in\Fix_{\G_n}(\sig)}|F_n(\eps)\cap(y+H_n)|
      \ge \max_{y\in\Fix_{\G_n}(\sig)}|F_n(\eps)\cap(y+\Fix_{\G_n}^{\circ}(\af))|
   \end{displaymath}
   and conclude from \eqref{eqn:separated} that
   \begin{equation}
      \label{eqn:coset-intersection}
      \lim_{n\to\infty}\frac1{|\zd/\G_n|}\,\log D_n(\eps)=0.
   \end{equation}
   
   Up to now $\eps$ was fixed, but now we start varying it. For every
   $n\ge1$, $F_n(\eps)$ has to intersect at least
   $|F_n(\eps)|/D_n(\eps)$ distinct cosets of
   $\Fix_{\G_n}^{\circ}(\af)$. By \eqref{eqn:liminf} and
   \eqref{eqn:coset-intersection},
   \begin{align*}
      \liminf_{n\to\infty}\frac1{|\zd/\G_n|}\,&\log|\Fix_{\G_n}(\af)/
      \Fix_{\G_n}^{\circ}(\af)|\\
      &\ge
      \sup_{\eps>0}\liminf_{n\to\infty}\frac1{|\zd/\G_n|}\,
      \log\bigl(|F_n(\eps)|/D_n(\eps)\bigr)\\
      &=\sup_{\eps>0}\liminf_{n\to\infty}\frac1{|\zd/\G_n|}\,\log|F_n(\eps)|=\h(\af). 
   \end{align*}
   Combining this with the previously mentioned fact that
   $\p^+(\af)=\h(\af)$ completes the proof of Theorem \ref{th:main}.
\end{proof}

\section{Diophantine consequences}
\label{sec:diophantine}

If $f\in\rd$ is atoral, then Theorem \ref{th:main} shows that the
Riemann sums for $\log|f|$ over $\OG$ converge to its integral $\m(f)$
as $\<\G\>\to\infty$. Roughly speaking, this means that those
$\bo\in\OG$ lying close to $\U(f)$ cannot contribute a disproportionate
amount to the Riemann sum, and this prevents these $\bo$ from getting
very close to $\U(f)$.

When $d=1$ this phenomenon was first established by Gelfond \cite{G}
using quite different methods. He showed that if $\lam\in\SS$ is an
algebraic number, then for all $\eps>0$ the inequality
\begin{equation}
   \label{eq:gelfond-inequality}
   |\lam^n-1|>e^{-\eps n}
\end{equation}
holds for all sufficiently large $n$. This can
be reformulated as follows. For $f\in R_1$ with $\U(f)\ne\emptyset$, for
all sufficiently large $n$ the inequality
\begin{equation}
   \label{eq:gelfond}
   \dist(\U(f),\Om_{n\ZZ})>e^{-\eps|\Om_{n\ZZ}|}
\end{equation}
holds. When $d\ge2$ our method gives a result formally identical to
this. However, it of course does not apply when $d=1$ since for this
case $f$ cannot be atoral unless $\U(f)=\emptyset$.

For $0\ne f\in\rd$ let us say that \emph{the Riemann sums for $\log|f|$
converge to $\m(f)$} provided that
\begin{displaymath}
   \frac1{|\OG|}\sum_{\bo\in\OG\smallsetminus\U(f)}
   \log|f(\bo)|\to\m(f)\coloneqq
   \int_{\sd}\log|f(\bs)|\,d\lam(\bs)
   \text{\quad as $\<\G\>\to\infty$}.
\end{displaymath}
By Theorem \ref{th:main}, this is valid for all atoral Laurent
polynomials.

We begin by establishing a quantitative consequence of the convergence
of the Riemann sums.

\begin{lemma}
   \label{lem:small-values}
   Let $0\ne f\in\rd$, and assume that the Riemann sums for $\log|f|$
   converge to $\m(f)$. Let $r_n>0$ and $\Gn\in\SF$ such that $r_n\to0$
   and $\<\Gn\>\to\infty$ as $n\to\infty$. Then
   \begin{equation}
      \label{eq:riemann-error}
      \frac1{|\OGn|}\sum_{\bo\in\OGn \atop 0<|f(\bo)|<r_n}\bigl|
      \log|f(\bo)|\bigr|
      \to 0 \text{\quad as $n\to\infty$}.
   \end{equation}
\end{lemma}

\begin{proof}
   Let $\eps>0$. For $r>0$ define
   $\phi_r(\bs)\coloneqq\max\{|f(\bs)|,r\}$ and
   $E_r(f)\coloneqq\{\bs\in\sd:0<|f(\bs)|<r\}$. We may assume throughout that
   $r_n,r<1$, so that $\log|f(\bo)|<0$ for all $\bo\in E_r(f)$.

   Note that $\log \phi_r$ is continuous on $\sd$, and so is integrable
   there. Furthermore, $\log|f|\in L^1(\sd,\lam)$, and $\log
   \phi_r\searrow\log|f|$. By the Monotone Convergence Theorem, we can
   find $r_0>0$ such that
   \begin{displaymath}      0\le\int_{\sd}\log\phi_r\,d\lam-\int_{\sd}\log|f|\,d\lam<\eps
   \end{displaymath}
   for all $0<r<r_0$. We may also choose $r_0$ small enough so that
   $\lam\bigl(E_{r_0}(f)\bigr)<\eps$.

   By our assumption about convergence of Riemann sums, we have that
   \begin{equation}
      \label{eq:estimate1}
      \Bigl| \m(f)- \frac1{|\OG|} \sum_{\bo\in\OG\smallsetminus\U(f)}
      \log|f(\bo)|
      \Bigr|<\eps
   \end{equation}
   for $\<\G\>$ sufficiently large. Since $\log \phi_{r_0}$ is
   continuous we also have that
   \begin{equation}
      \label{eq:estimate2}
       \Bigl| \int\log\phi_{r_0}\,d\lam - \frac1{|\OG|} \sum_{\bo\in\OG}
      \log|\phi_{r_0}(\bo)|
      \Bigr|<\eps
   \end{equation}
   for sufficiently large $\<\G\>$. Finally, a standard argument using
   upper and lower approximations of the indicator of function of $E_{r_0}(f)$ by
   continuous functions shows that
   \begin{equation}
      \label{eq:estimate3}
      \frac{|E_{r_0}(f)\cap\OG|}{|\OG|}<\eps
   \end{equation}
   for all sufficiently large $\<\G\>$. Thus we may choose $L_0$ large
   enough so that \eqref{eq:estimate1},  \eqref{eq:estimate2}, and
   \eqref{eq:estimate3} hold if $\<\G\> >L_0$.

   Now assume that $r_n\to0$ and $\<\Gn\>\to\infty$. Choose $n_0$ such
   that $r_n<r_0$ and $\<\Gn\> >L_0$ for all $n\ge n_0$. Then since
   $E_{r_n}(f)\subset E_{r_0}(f)$ and $\log{f(\bs)}<0$ for all $\bs\in
   E_{r_0}(f)$, it follows that
   \begin{align*}
      0 &\le\frac1{|\OGn|} \sum_{\bo\in\OGn\cap E_{r_n}(f)}\bigl|
          \log|f(\bo)|\bigr| 
        \le \frac1{|\OGn|} \sum_{\bo\in\OGn\cap E_{r_0}(f)}
           - \log|f(\bo)| \\
        &= \frac1{|\OGn|}\sum_{\bo\in\OGn}\log\phi_{r_0}(\bo)
           -\frac{(\log r_0)|\OGn\cap
           E_{r_0}(f)|}{|\OGn|}-\sum_{\bo\in\OGn\smallsetminus\U(f)}
           \log|f(\bo)| \\
        &\le \m(\phi_{r_0})+\eps-\eps\log
        r_0-(\m(f)-\eps)<3\eps+\eps\log r_0.
   \end{align*}
   Since $\eps>0$ was arbitrary, this completes the proof.
\end{proof}

To make use of this result, we introduce two counting functions. For
$r>0$ let
\begin{displaymath}
   M_f(\OG,r)\coloneqq |\{\bo\in\OG:0<\dist(\bo,\U(f))<r\}|,
\end{displaymath}
and
\begin{displaymath}
   N_f(\OG,r) \coloneqq |\{\bo\in\OG:0<|f(\bo)|<r\}|.
\end{displaymath}
Observe that $f$ is Lipschitz on $\sd$, say with Lipschitz constant
$K$. Then $M_f(\OG,r)\le N_f(\OG,Kr)$.

\begin{theorem}
   \label{th:quantitative-diophantine}
   Let $0\ne f\in \rd$, and assume that the Riemann sums for $\log|f|$
   converge to $\m(f)$. Let $r_n\to0$ and $\<\Gn\>\to\infty$ as
   $n\to\infty$. Then
   \begin{displaymath}
      \frac{M_f(\OGn,r_n)\cdot\log(1/r_n)}{|\OGn|}\to0
      \text{\quad as $n\to\infty$}.
   \end{displaymath}
   In particular, for every $\eps>0$ there is an $n_0$ such that
   \begin{displaymath}
      M_f(\OGn,r_n)<\frac{\eps|\OGn|}{\log(1/r_n)}
      \text{\quad for $n\ge n_0$}.
   \end{displaymath}
\end{theorem}

\begin{proof}
   Let $\eps>0$, and $K\ge1$ be a Lipschitz constant for $f$ on
   $\sd$. By Lemma \ref{lem:small-values}, for all sufficiently large
   $n$ we have that
   \begin{displaymath}
      \frac1{|\OGn|}N_f(\OGn,Kr_n)\log\frac1{r_n}\le
      \frac1{|\OGn|}\sum_{\bo\in\OGn\atop
      0<|f(\bo)|<Kr_n}\bigl|\log|f(\bo)|\bigr|<\eps .
   \end{displaymath}
   Since $M_f(\OGn,r_n)\le N_f(\OGn,Lr_n)$ and
   $\log(1/Lr_n)\le\log(1/r_n)$, the result follows.
\end{proof}

\begin{corollary}
   \label{cor:lowerbound}
   Let $f\in\rd$ be an atoral Laurent polynomial and $\eps>0$. Then
   \begin{displaymath}
      \dist(\OG,\U(f)\smallsetminus\OG)>e^{-\eps|\OG|}
   \end{displaymath}
   whenever $\<\G\>$ is sufficiently large.
\end{corollary}

\begin{proof}
   Suppose there is an $\delta_0>0$ and a sequence $\Gn\in\SF$ with
   $\<\Gn\>\to\infty$ such that
     \begin{displaymath}
        \dist(\OGn,\U(f)\smallsetminus\OGn)<e^{-\delta_0|\OGn|}
     \end{displaymath}
   for all $n\ge1$. Put $r_n=e^{-\delta_0|\OGn|}\to0$. By the previous
   theorem, with $\eps=\delta_0/2$, for sufficiently large $n$ we would
   have
   \begin{displaymath}
      1\le M_f(\OGn,r_n)<\frac{\delta_0}{2}\frac{|\OGn|}{\delta_0|\OGn|}=\frac12,
   \end{displaymath}
   which is impossible.
\end{proof}

We can use Theorem \ref{th:quantitative-diophantine} to obtain slightly
weaker results even when $f$ is toral (e.g., in Gelfond's original
setting). Let $\ZZ_+$ denote $\{0,1,2,\dots\}$.

\begin{theorem}
   \label{th:weaker}
   Let $d\ge1$ and $f\in\rd$ be a nonzero irreducible Laurent
   polynomial. For every $\eps>0$ and every $\psi\colon\ZZ_+\to\ZZ_+$
   with $\psi(n)\to\infty$ as $n\to\infty$ (no matter how slowly),
   there is an $L\ge1$ such that for every $\G\in\SF$ with $\<\G\> >L$
   and every $\bo\in\OG\smallsetminus\U(f)$
   we have that
   \begin{equation}
      \label{eq:distance-bound}
      \dist(\bo,\U(f))\ge e^{-\eps\psi(\<\G\>)|\OG|}.
   \end{equation}
\end{theorem}

\begin{proof}
   Let $g(u)=u-1$, and consider the Laurent polynomial
   \begin{displaymath}
      h(u_1,\dots,u_d,u_{d+1})=f(u_1,\dots,u_d)f^*(u_1,\dots,u_d)+g(u_{d+1})g^*(u_{d+1})
      \in R_{d+1}.
   \end{displaymath}
   Clearly $\U(h)=\U(f)\times\{1\}$, and so $h$ is atoral.

   Let $\psi\colon\ZZ_+\to\ZZ_+$ with $\psi(n)\to\infty$ as
   $n\to\infty$. For every finite index subgroup $\G$ of $\zd$ and every
   $k\ge1$, let $\G^{(k)}$ be the finite-index subgroup of $\ZZ^{d+1}$
   generated by $\G\times\{0\}$ and $(\bzero,\psi(k))$. Then
   $|\ZZ^{d+1}/\G^{(k)}|=|\zd/\G|\psi(k)$ and
   $\<\G^{(k)}\>=\min\{\<\G\>,\psi(k)\}$. We can then apply Corollary
   \ref{cor:lowerbound} to obtain \eqref{eq:distance-bound}.
\end{proof}

Applying the previous theorem when $d=1$, we obtain the following weaker
version of Gelfond's result.

\begin{corollary}
   \label{cor:weak-gelfond}
   Let $\gamma\in\SS$ be an algebraic number which is not a root of
   unity. Then for every $\psi\colon\ZZ_+\to\ZZ_+$ with
   $\psi(n)\to\infty$ as $n\to\infty$ and every $\eps>0$, there is a $n_0$ such that for
   every $n\ge n_0$ and every $n$th root of unity $\om\in\Om_{n\ZZ}$ we
   have that
   \begin{displaymath}
      \dist(\gamma,\om)\ge e^{-\eps\psi(n)n}.
   \end{displaymath}
\end{corollary}

\begin{remark}
   In Corollary \ref{cor:weak-gelfond} we allow $\psi(n)$ to grow
   arbitrarily slowly. But we still require $\psi(n)\to\infty$, so that
   this is not quite strong enough to obtain Gelfond's result
   \eqref{eq:gelfond-inequality}. Thus this also does not imply the
   convergence of the logarithmic growth rate of periodic points for
   quasihyperbolic toral automorphisms. For this one really needs
   \eqref{eq:gelfond-inequality}.
\end{remark}

\end{document}